\DeclareSymbolFont{AMSb}{U}{msb}{m}{n}
\documentclass[10pt,a4paper,leqno,noamsfonts]{amsart}
   \makeatletter 
   \renewcommand\@biblabel[1]{#1.}
   \makeatother
\usepackage[english]{babel}
\usepackage[dvipsnames]{xcolor}
\usepackage{graphicx,pifont}
\usepackage[utopia]{mathdesign}
\usepackage[a4paper,top=3cm,bottom=3cm,left=3cm,right=3cm,marginparwidth=60pt]{geometry}
\usepackage[utf8]{inputenc}
\usepackage{braket,caption,comment,mathtools,stmaryrd}
\usepackage[usestackEOL]{stackengine}
\usepackage{multirow,booktabs,microtype,relsize}
\usepackage[colorlinks,bookmarks]{hyperref} %
   \definecolor{britishracinggreen}{rgb}{0.0, 0.26, 0.15}
   \definecolor{cobalt}{rgb}{0.0, 0.28, 0.67}
   \definecolor{cornellred}{rgb}{0.7, 0.11, 0.11}
      \hypersetup{colorlinks,%
            citecolor=britishracinggreen,%
            filecolor=black,%
            linkcolor=cobalt,%
            urlcolor=cornellred}
      \numberwithin{equation}{section}

\usepackage{braket,caption,comment,mathtools,stmaryrd}
\DeclareSymbolFont{usualmathcal}{OMS}{cmsy}{m}{n}
\DeclareSymbolFontAlphabet{\mathcal}{usualmathcal}

\newcommand{\BA}{{\mathbb{A}}}
\newcommand{\BN}{{\mathbb{N}}}
\newcommand{\BC}{{\mathbb{C}}}
\newcommand{\BL}{{\mathbb{L}}}
\newcommand{\BZ}{{\mathbb{Z}}}
\newcommand{\BR}{{\mathbb{R}}}
\newcommand{\CM}{{\mathcal M}}
\newcommand{\CT}{{\mathcal T}}

\newcommand{\dd}{\mathrm{d}}
\newcommand{\crit}{\operatorname{crit}}
\newcommand{\zetass}{\zeta\textrm{-ss}}
\newcommand{\zetast}{\zeta\textrm{-st}}

\newcommand\Hom{\operatorname{Hom}}

\newcommand\Spec{\operatorname{Spec}}

\newcommand{\PT}{\mathsf{PT}}
\newcommand{\DT}{\mathsf{DT}}
\newcommand{\OO}{\mathscr O}

\DeclareMathOperator{\points}{points}
\DeclareMathOperator{\Hilb}{Hilb}

\DeclareMathOperator{\Quot}{Quot}
\DeclareMathOperator{\con}{con}

\DeclareMathOperator{\Sym}{Sym}
\DeclareMathOperator{\Coh}{Coh}
\DeclareMathOperator{\Tr}{Tr}
\DeclareMathOperator{\vir}{\mathrm{vir}}

\DeclareMathOperator{\Var}{Var}
\DeclareMathOperator{\Rep}{R}
\DeclareMathOperator{\GL}{GL}

\usepackage[all]{xy}
\usepackage{tikz}
\usepackage{tikz-cd}
\usepackage{adjustbox}
\usepackage{rotating}
\usepackage{comment}

\usetikzlibrary{matrix,shapes,intersections,arrows,decorations.pathmorphing}
\tikzset{commutative diagrams/arrow style=math font}
\tikzset{commutative diagrams/.cd,
mysymbol/.style={start anchor=center,end anchor=center,draw=none}}

\tikzset{
shift up/.style={
to path={([yshift=#1]\tikztostart.east) -- ([yshift=#1]\tikztotarget.west) \tikztonodes}
}}

\theoremstyle{definition}
\newtheorem{definition}{Definition}[section]

\newtheorem{convention}{Convention}

\newtheorem{remark}[definition]{Remark}

\newtheoremstyle{thm} % <name> % (ambienti con dimostrazione)
        {3mm}% <Space above>
        {3mm}% <Space below>
        {\slshape}% <Body font> % 
        {0mm}% <Indent amount>
        {\bfseries}% <Theorem head font>
        {.}% <Punctuation after theorem head>
        {1mm}% <Space after theorem head>
        {}% <Theorem head spec (can be left empty, meaning 'normal')> 
\theoremstyle{thm}
\newtheorem{theorem}[definition]{Theorem}
\newtheorem{corollary}[definition]{Corollary}
\newtheorem{lemma}[definition]{Lemma}

\newtheorem{thm}{Theorem}

\title[Framed motivic DT invariants of small crepant resolutions]{Framed motivic Donaldson--Thomas invariants of small crepant resolutions}
\author[Alberto Cazzaniga and Andrea T. Ricolfi]{Alberto Cazzaniga \and Andrea T. Ricolfi}

\begin{document}

\begin{abstract}{For an arbitrary integer $r\geq 1$, we compute $r$-framed motivic DT and PT invariants of small crepant resolutions of toric Calabi--Yau $3$-folds, establishing a ``higher rank'' version of the motivic DT/PT wall-crossing formula. This generalises the work of Morrison and Nagao. Our formulae, in particular their relationship with the $r=1$ theory, fit nicely in the current development of higher rank refined DT invariants.}
\end{abstract}

\maketitle

{\hypersetup{linkcolor=black}
\tableofcontents}

\section{Introduction}
Let $Y$ be a smooth Calabi--Yau $3$-fold. Donaldson--Thomas (DT in short) theory in rank $1$ is an enumerative theory virtually enumerating curves embedded in $Y$. The moduli space being `enumerated' is the Hilbert scheme of $1$-dimensional subschemes of $Y$. On the other hand, Pandharipande--Thomas (PT in short) theory has as its main character the moduli space of (rank $1$) stable pairs on $Y$, which are pairs $(F,s)$ where $F \in \Coh Y$ is a purely $1$-dimensional sheaf and $s\colon \OO_Y \to F$ is a section with $0$-dimensional cokernel. Both enumerative theories admit motivic refinements; in general it is very hard to produce explicit formulae for the generating functions of the motivic DT and PT invariants, but when the moduli spaces in question admit a description in terms of stable representations of the Jacobi algebra of a quiver with potential $(Q,\omega)$, the problem might become more tractable.
For instance, Morrison and Nagao computed in \cite{MorrNagao} motivic DT and PT invariants of small crepant resolutions $Y_\sigma$ of the affine toric Calabi--Yau $3$-fold
\[
X = \Spec \BC[x,y,z,w]/(xy-z^{N_0}w^{N_1}) \,\subset\,\BA^4,
\]
generalising previous results on the resolved conifold \cite{RefConifold}, corresponding to the case $N_0=N_1=1$.
Such resolutions $Y_\sigma \to X$ are indexed by \emph{partitions} $\sigma$ of a polygon $\Gamma_{N_0,N_1}$ naturally attached to $X$ (more details in \S\,\ref{sec:NCCR}). Each partition $\sigma$ defines a quiver with potential $(Q_{\sigma},\omega_\sigma)$ with $N=N_0+N_1$ vertices (see Figure \ref{fig:quiver_of_partition} for an example of such a $Q_\sigma$), and for any $r\geq 1$ one can consider the $r$-\emph{framed} quiver (Definition \ref{def:r-framing}) with potential $(\widetilde{Q}_{\sigma},\omega_{\sigma})$. We denote by $\widetilde J_\sigma$ the corresponding \emph{Jacobi algebra}. A generic choice of stability parameters $\zeta_{\PT}$ and $\zeta_{\DT}$, respectively in the PT and DT regions of the space of all stability parameters of $Q_\sigma$, gives rise to generating functions 
\[
\PT_r(Y_\sigma;s,T)\,\,\,\textrm{and}\,\,\,\DT_r(Y_\sigma;s,T)
\]
of motivic invariants, where (at least in the $r=1$ case) $s$ represents the point class and $T$ is a vector of curve classes. The definition of the series $\PT_r$ and $\DT_r$ is as follows. One first sets, for a generic stability parameter $\zeta$,
\[
\mathsf Z_{\zeta}(y_0,y_1,\ldots,y_{N-1}) = \sum_{\alpha \in \BN^{(Q_\sigma)_0}} \,\left[\mathfrak M_\zeta(\widetilde J_\sigma,\alpha) \right]_{\vir}\cdot y^{\alpha}
\]
where the \emph{virtual motive} $[\,\cdot\,]_{\vir}$ of the moduli stack $\mathfrak M_\zeta(\widetilde J_\sigma,\alpha)$ of $\zeta$-stable $\widetilde{J}_{\sigma}$-modules with dimension vector $(\alpha,1)$ is introduced in Definition \ref{def:motivic_partition_functions}.
One then defines
\begin{equation}\label{PT-DT-series}
\begin{split}
\PT_r(Y_\sigma;s,T) &= \mathsf Z_{\zeta_{\PT}}(s,T_1,\ldots,T_{N-1}) \\ \DT_r(Y_\sigma;s,T) &= \mathsf Z_{\zeta_{\DT}}(s,T_1,\ldots,T_{N-1})
\end{split}
\end{equation}
where $s=y_0y_1\cdots y_{N-1}$, $T_i = y_i^{-1}$ and $T=(T_1,\ldots,T_{N-1})$.

The generating functions \eqref{PT-DT-series} are computed explicitly for $r=1$ in \cite[Cor.~0.3]{MorrNagao}. The result, recalled in \S\,\ref{subsec:Computing_invariants}, is the following: one has
\[
\PT_1(Y_\sigma;s,T) = \prod_{1\leq a\leq b\leq N-1} Z_{[a,b]}(s,T_a\cdots T_b),
\]
where, letting $\set{C_i|1\leq i\leq N-1}$ be the set of components of the exceptional curve and $c(a,b)$ the number of $(-1,-1)$-curves in $\set{C_i|a\leq i\leq b}$, one sets
\[
Z_{[a,b]}(s,T_a\cdots T_b) = 
\begin{cases}
\displaystyle\prod_{m\geq 1}\prod_{j=0}^{m-1}\left(1-\BL^{j+\frac{1}{2}-\frac{m}{2}}(-s)^mT_a\cdots T_b \right) & \textrm{if }c(a,b)\textrm{ is odd} \\ \\
\displaystyle\prod_{m\geq 1}\prod_{j=0}^{m-1}\left(1-\BL^{j+1-\frac{m}{2}}(-s)^mT_a\cdots T_b \right)^{-1} & \textrm{if }c(a,b)\textrm{ is even}.
\end{cases}
\]
As for the DT series in rank $1$, one has the DT/PT correspondence
\[
\DT_1(Y_\sigma;s,T)=\DT_1^{\points}(Y_\sigma,s)\cdot\PT_1(Y_\sigma;s,T),
\]
where $\DT_1^{\points}(Y_\sigma,s)$ is the Behrend--Bryan--Szendr\H{o}i generating function \cite{BBS}, that we recall in \eqref{eqn:Hilb_points_Ysigma}.

The goal of this paper is to compute the generating functions $\PT_r(Y_\sigma;s,T)$ and $\DT_r(Y_\sigma;s,T)$ for arbitrary $r$. The result, as we will show, is a full factorisation of the above series as $r$-fold (twisted) products of the $r=1$ generating functions. Moreover, we establish an $r$-framed version of the motivic DT/PT correspondence for $Y_\sigma$. 

Our main result, proved in \S\,\ref{subsec:Computing_invariants}, is the following.

\begin{thm}\label{mainthm}
Let $Y_\sigma$ be the crepant resolution of $X$ corresponding to a partition $\sigma$. There are factorisations
\begin{equation}
\label{eqn:main_formulae}
\begin{split}
\PT_r(Y_\sigma;s,T) &\,\,=\,\, \prod_{i=1}^r \PT_1\left(Y_\sigma;(-1)^{r+1}\BL^{\frac{-r-1}{2}+i}s,T\right),\\
\DT_r(Y_\sigma;s,T) &\,\,=\,\, \prod_{i=1}^r \DT_1\left(Y_\sigma;(-1)^{r+1}\BL^{\frac{-r-1}{2}+i}s,T\right).
\end{split}
\end{equation}
Furthermore, the $r$-framed motivic DT/PT correspondence holds: there is an identity
\[
\DT_r(Y_\sigma;s,T)\,\,=\,\,\DT_r^{\points}(Y_\sigma,s)\cdot\PT_r(Y_\sigma;s,T),
\]
where $\DT_r^{\points}(Y_\sigma,s)$ is the virtual motivic partition function of the Quot scheme of points on $Y_\sigma$.
\end{thm}

The series $\DT_r^{\points}(\BA^3,s) = \sum_{n \geq 0} [\Quot_{\BA^3}(\OO^{\oplus r},n)]_{\vir}\cdot s^n$, originating from the critical locus structure on $\Quot_{\BA^3}(\OO^{\oplus r},n)$, is studied in detail in \cite{Cazzaniga_Thesis,ThesisR,refinedDT_asymptotics}. The series $\DT_r^{\points}(Y,s)$ was introduced and computed for all $3$-folds $Y$ in \cite[\S\,4]{Quot19}, generalising the $r=1$ case corresponding to $\Hilb^nY$ \cite{BBS}. See \S\,\ref{sec:motive_quot} for more details --- for instance, an explicit formula for $\DT_r^{\points}(Y_\sigma,s)$ will be given in Equation \eqref{DT_points_Y_sigma}. 

A first instance of Formulae \eqref{eqn:main_formulae} was computed in \cite[Chap. 3]{Cazzaniga_Thesis} for the case of the resolved conifold and the resolution of a line of ${A}_{2}$ singularities.

The same factorisation of generating functions of ``rank $r$ objects'' into $r$ copies of generating functions of rank $1$ objects, shifted precisely as in Formulae \eqref{eqn:main_formulae}, has recently been observed in the context of higher rank K-theoretic DT invariants \cite{FMR_K-DT} and in string theory \cite{Magnificent_colors}. 

Even though the geometric meaning of the moduli spaces of quiver representations giving rise to the $r$-framed invariants \eqref{eqn:main_formulae}, for arbitrary $r$, is not as clear as in the  $r=1$ case, we do believe that such moduli spaces have a sensible geometric interpretation as suitable ``higher rank'' analogues of the Hilbert scheme of curves in $Y_\sigma$ (DT side) and the moduli space of stable pairs on $Y_\sigma$ (PT side). We come back to this in Remark \ref{remark:geometric_interpretation}, where we discuss a geometric interpretation of the framed moduli spaces in the PT chamber for the case of the conifold and $\widetilde{A}_{2}$ quivers.

%%%%%%%%%%%%%%%%%%%%%%%%%%%%%%%%%
\section{Background material}
\label{sec:background_material}

%%%%%%%%%%%%%%%%%%%%%%%%%%%%%%%%%
\subsection{Rings of motives}
\label{subsec:motivic_quantum_torus}

In this subsection we recall the definitions of various rings where the motivic invariants we want to study live.

As in \cite{RefConifold,MorrNagao}, we let $\CM_{\BC}$ be the Grothendieck ring of the category of effective Chow motives over $\BC$ with rational coefficients \cite{Manin_motivic}, extended with $\BL^{-1/2}$. A lambda-ring structure on $\CM_{\BC}$ is obtained by setting $\sigma_n ([X]) = [\Sym^nX]$ and $\sigma_n (\BL^{1/2})=\BL^{n/2}$ to define the lambda operations. In particular, there is a well defined notion of power structure and plethystic exponential on $\CM_{\BC}$ (see e.g.~\cite[\S\,2.5]{BBS} or \cite[\S\,1.5.1]{DavisonR} for their formal properties). We consider the dimensional completion \cite{Behrend_Ajneet}
\[
\widetilde{\CM}_{\BC} = \CM_{\BC} \llbracket \BL \rrbracket,
\]
which is also a lambda-ring, and in which the motives $[\GL_k]$ of all general linear groups are invertible.

%%%%%%%%%%%%%%%%%%%%%%%%%%%%%%%
\subsubsection{The virtual motive of a critical locus}

Let $U$ be a smooth $d$-dimensional $\BC$-scheme, $f\colon U \to \BA^1$ a regular function. The \emph{virtual motive} of the critical locus $\crit(f) = Z( \dd f) \subset U$, depending on the pair $(U,f)$, is defined in \cite{RefConifold,MorrNagao} as the motivic class
\[
\bigl[\crit(f)\bigr]_{\vir} = -(-\BL^{\frac{1}{2}})^{-d}\cdot \left[\phi_f\right] \,\in \,\CM_{\BC}^{\hat\mu},
\]
where $[\phi_f] \in K_0^{\hat\mu}(\Var_{\BC})$ is the (absolute) motivic vanishing cycle class defined by Denef and Loeser \cite{DenefLoeser1} and the `$\hat\mu$' decoration refers to $\hat\mu$-equivariant motives, where $\hat\mu$ is the group of all roots of unity. However, all the motivic invariants studied here will live in the subring $\CM_{\BC}\subset \CM_{\BC}^{\hat\mu}$ of classes carrying the trivial $\hat\mu$-action, so we will not be concerned with the subtle structure of this larger ring. 

As an example, consider the function $f = 0 \in \Gamma(U)$. Then $\crit(f) = U$ and $[\phi_f] = -[U]$, so $[U]_{\vir} = (-\BL^{\frac{1}{2}})^{-\dim U}\cdot [U]$. For instance,

\begin{equation}\label{example:vir_motive_smooth_scheme}
[\GL_k]_{\vir} = (-\BL^{\frac{1}{2}})^{-k^2}\cdot [\GL_k].
\end{equation}
%\end{example}

\begin{remark}\label{remark:different_conventions}
Our definition of $[\crit(f)]_{\vir}$ differs from the original one \cite[\S\,2.8]{BBS}, which is also the one used in \cite{DavisonR,refinedDT_asymptotics}. We decided to adopt the conventions in \cite{RefConifold,MorrNagao} to keep close to the original formulae. In practice, the difference amounts to the substitution $\BL^{1/2} \leftrightarrow -\BL^{1/2}$. In particular, the Euler number specialisation with our conventions is $\BL^{1/2}\to 1$, instead of $\BL^{1/2}\to -1$.
\end{remark}

%%%%%%%%%%%%%%%%%%%%%%%%%%%%%%%%%
\subsection{Quivers: framings, and motivic quantum torus}
A quiver $Q$ is a finite directed graph, determined by its sets $Q_0$ and $Q_1$ of vertices and edges, respectively, along with the maps $h$, $t\colon Q_1 \to Q_0$ specifying where an edge starts or ends.
We use the notation
\[
\begin{tikzcd}
t(a) \,\,\bullet  \arrow{rr}{a} & & \bullet\,\, h(a)
\end{tikzcd}
\]
to denote the \emph{tail} and the \emph{head} of an edge $a \in Q_1$.

All quivers in this paper will be assumed connected.
The \emph{path algebra} $\BC Q$ of a quiver $Q$ is defined, as a $\BC$-vector space, by using as a $\BC$-basis the set of all paths in the quiver, including a trivial path $e_i$ for each $i \in Q_0$. The product is defined by concatenation of paths whenever the operation is possible, and $0$ otherwise. The identity element is $\sum_{i \in Q_0}e_i \in \BC Q$.

On a quiver $Q$ one can define the \emph{Euler--Ringel form} $\chi(-,-)\colon \BZ^{Q_0}\times \BZ^{Q_0} \to \BZ$ by
\[
\chi(\alpha,\beta) = \sum_{i \in Q_0}\alpha_i\beta_i - \sum_{a \in Q_1}\alpha_{t(a)}\beta_{h(a)},
\]
as well as the skew-symmetric form
\[
\braket{\alpha,\beta} = \chi(\alpha,\beta)-\chi(\beta,\alpha).
\]

The following construction will be central in our paper.

\begin{definition}[$r$-framing]\label{def:r-framing}
Let $Q$ be a quiver with a distinguished vertex $0\in Q_0$, and let $r$ be a positive integer. We define the quiver $\widetilde Q$ by adding one vertex, labelled $\infty$, to the original vertices in $Q_0$, and $r$ edges $\infty\to 0$. We refer to $\widetilde Q$ as the $r$-\emph{framed} quiver obtained out of $(Q,0)$.
\end{definition}

The $r$-framing construction was applied to the $3$-loop quiver (on the left in Figure \ref{fig:3loop_and_conifold_quiver}) in \cite{Cazzaniga_Thesis,ThesisR,BR18,refinedDT_asymptotics}, following the $r=1$ case studied by Behrend--Bryan--Szendr\H{o}i \cite{BBS}, and to the conifold quiver (on the right in Figure \ref{fig:3loop_and_conifold_quiver}) in \cite{Cazzaniga_Thesis}. In this paper, it will be applied more generally to the quivers arising in the work of Morrison--Nagao \cite{MorrNagao}, which we briefly discuss in \S\,\ref{sec:NCCR}. The case $r=1$ was covered in \cite{RefConifold,MorrNagao}.

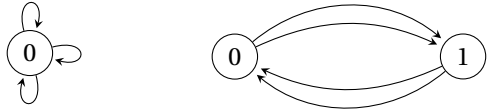
\begin{figure}[ht]
\centering
\begin{tikzpicture}[>=stealth,->,shorten >=2pt,looseness=.5,auto]
  \matrix [matrix of math nodes,
           column sep={3cm,between origins},
           row sep={3cm,between origins},
           nodes={circle, draw, minimum size=3.5mm}]
{ 
|(B)| 0 \\         
};
\tikzstyle{every node}=[font=\small\itshape]
\path[->] (B) edge [loop above] node {} ()
              edge [loop right] node {} ()
              edge [loop below] node {} ();
\end{tikzpicture}
\qquad \qquad
\begin{tikzpicture}[>=stealth,->,shorten >=2pt,looseness=.5,auto]
  \matrix [matrix of math nodes,
           column sep={3cm,between origins},
           row sep={3cm,between origins},
           nodes={circle, draw, minimum size=3.5mm}]
{ 
|(A)| 0 & |(B)| 1 \\         
};
\tikzstyle{every node}=[font=\small\itshape]
\draw (A) to [bend left=25,looseness=1] (B) node {};
\draw (A) to [bend left=40,looseness=1] (B) node {};
\draw (B) to [bend left=25,looseness=1] (A) node {};
\draw (B) to [bend left=40,looseness=1] (A) node {};
\end{tikzpicture}
\caption{The $3$-loop quiver $L_3$ and the conifold quiver $Q_{\con}$.}\label{fig:3loop_and_conifold_quiver}
\end{figure}

Let $Q$ be a quiver. Define its \emph{motivic quantum torus} (or \emph{twisted motivic algebra}) as
\[
\mathcal T_Q = \prod_{\alpha \in \BN^{Q_0}} \widetilde{\mathcal{M}}_{\BC}\cdot y^\alpha
\]
with product rule
\begin{equation}\label{eqn:product_in_TQ}
    y^\alpha\cdot y^\beta = (-\BL^{\frac{1}{2}})^{\braket{\alpha,\beta}}y^{\alpha+\beta}.
\end{equation}
If $\widetilde{Q}$ is the $r$-framed quiver associated to $(Q,0)$ via Definition \ref{def:r-framing}, one has that $\mathcal T_Q$ sits inside $\mathcal T_{\widetilde{Q}}$ as a $\widetilde{\CM}_{\BC}$-subalgebra, and there is a $\BZ$-module decomposition 
\[
\mathcal T_{\widetilde{Q}} = \mathcal T_Q \oplus \prod_{d\geq 0} \widetilde{\CM}_{\BC}\cdot y_\infty^d,
\]
where we have set $y_\infty = y^{(\mathbf 0,1)}$. Similarly, a generator $y^\alpha \in \mathcal T_Q$ will be identified with its image $y^{(\alpha,0)} \in \mathcal T_{\widetilde{Q}}$.

%%%%%%%%%%%%%%%%%%%%%%%%%%%%%%%%%%%%%%%%%%%%%%%%%%%%%%%%%%%%%%%%%%
\subsection{Quiver representations and their stability}

Let $Q$ be a quiver.
A \emph{representation} $\rho$ of $Q$ is the datum of a finite dimensional $\BC$-vector space $\rho_i$ for every vertex $i\in Q_0$, and a linear map $\rho(a)\colon \rho_i\to \rho_j$ for every edge $a\colon i\to j$ in $Q_1$.
The \emph{dimension vector} of $\rho$ is the vector $\underline{\dim}\,\rho = (\dim_{\BC}\rho_i)_i\in \mathbb N^{Q_0}$, where $\BN = \BZ_{\geq 0}$.

\begin{convention}\label{convention:dimvec}
Let $Q$ be a quiver, $\widetilde Q$ the associated $r$-framed quiver. The dimension vector of a representation $\widetilde \rho$ of $\widetilde Q$ will be denoted $(\alpha,d)$, where $\alpha \in \BN^{Q_0}$ and $\dim_{\BC}\widetilde{\rho}_\infty = d \in \BN$.
\end{convention}

Representations of a quiver $Q$ form an abelian category, which is equivalent to the category of left modules over the path algebra $\BC Q$ of the quiver.
The space of all representations of $Q$, with a fixed dimension vector $\alpha\in \mathbb N^{Q_0}$, is the affine space
\[
\Rep(Q,\alpha) = \prod_{a \in Q_1}\Hom_{\BC}(\BC^{\alpha_{t(a)}},\BC^{\alpha_{h(a)}}).
\]
The gauge group $\GL_\alpha = \prod_{i\in Q_0} \GL_{\alpha_i}$ acts on $\Rep(Q,\alpha)$ by $(g_i)_i \cdot (\rho(a))_{a\in Q_1} = (g_{h(a)}\circ\rho(a)\circ g_{t(a)}^{-1})_{a \in Q_1}$. The quotient stack 
\[
\mathfrak M(Q,\alpha) = \bigl[\Rep(Q,\alpha)/\GL_\alpha\bigr]
\]
parametrises isomorphism classes of representations of $Q$ with dimension vector $\alpha$.

Following \cite{RefConifold,MorrNagao}, we recall the notion of (semi)stability of a representation.

\begin{definition}\label{centralcharge}
A \emph{central charge} is a group homomorphism $\mathrm{Z}\colon \mathbb Z^{Q_0}\to \BC$ such that the image of $\mathbb N^{Q_0}\setminus 0$ lies inside $\mathbb H_+ = \set{t e^{\sqrt{-1}\pi\varphi}|t>0,\,0<\varphi\leq 1}$. For every $\alpha\in \mathbb N^{Q_0}\setminus 0$, we denote by $\varphi(\alpha)$ the real number $\varphi$ such that $\mathrm{Z}(\alpha) = t e^{\sqrt{-1}\pi\varphi}$. It is called the \emph{phase} of $\alpha$ with respect to $\mathrm{Z}$.
\end{definition}

Note that every vector $\zeta\in \BR^{Q_0}$ induces a central charge $\mathrm{Z}_{\zeta}$ if we set
$\mathrm{Z}_{\zeta}(\alpha) = -\zeta\cdot \alpha + \lvert\alpha\rvert \sqrt{-1}$,
where $\lvert\alpha\rvert = \sum_{i\in Q_0}\alpha_i$. We denote by $\varphi_\zeta$ the induced phase function, and we set $
\varphi_\zeta(\rho) = \varphi_\zeta(\underline{\dim}\,\rho)$ for every representation $\rho$ of $Q$. The \emph{slope function} attached to $\mathrm{Z}_\zeta$ assigns to $\alpha \in \BN^{Q_0}\setminus 0$ the real number $\mu_\zeta(\alpha) = \zeta\cdot \alpha / \lvert \alpha\rvert$. Note that $\varphi_\zeta(\alpha) < \varphi(\beta)$ if and only if $\mu_\zeta(\alpha) < \mu_\zeta(\beta)$ (cf.~\cite[Rem.~3.5]{MorrNagao}).

\begin{definition}\label{stablereps}
Fix $\zeta\in \BR^{Q_0}$. A representation $\rho$ of $Q$ is called \emph{$\zeta$-semistable} if 
\[
\varphi_\zeta(\rho')\leq \varphi_\zeta(\rho)
\]
for every nonzero proper subrepresentation $0\neq \rho'\subsetneq \rho$. When strict inequality holds, we say that $\rho$ is \emph{$\zeta$-stable}. Vectors $\zeta\in \BR^{Q_0}$ are referred to as \emph{stability parameters}. 
\end{definition}

For a fixed $\zeta$, every representation $\rho$ admits a unique filtration
\[
\mathrm{HN}_\zeta(\rho)\,\,\,\colon \qquad 0 = \rho_0 \subset \rho_1 \subset \cdots \subset \rho_s = \rho,
\]
called the \emph{Harder--Narasimhan filtration}, such that $\rho_i/\rho_{i-1}$ is $\zeta$-semistable for $1\leq i\leq s$, and there are strict inequalities $\varphi_\zeta(\rho_1/\rho_0)>\varphi_\zeta(\rho_2/\rho_1)>\cdots > \varphi_\zeta(\rho/\rho_{s-1})$. 

\begin{remark}
The existence, uniqueness and functoriality of the Harder--Narasimhan filtration yields a stratification of the moduli stack of all $Q$-representations into locally closed substacks, indexed by Harder--Narasimhan type (this is a direct consequence of \cite[Prop.~3.4]{Reineke_inventiones}); this stratification induces  relations in the motivic quantum torus, which are implicitly used in Lemma \ref{dec2}.
\end{remark}

\begin{definition}[{\cite[\S\,1.3]{RefConifold}}]
Let $\alpha \in \BN^{Q_0}$ be a dimension vector. A stability parameter $\zeta$ is called $\alpha$-\emph{generic} if for any $0<\beta<\alpha$ one has $\varphi_\zeta(\beta) \neq \varphi_\zeta(\alpha)$.
\end{definition}

The sets of $\zeta$-stable and $\zeta$-semistable representations with given dimension vector $\alpha$ form a chain of open subsets 
\[
\Rep^{\zetast}(Q,\alpha)\subset \Rep^{\zetass}(Q,\alpha)\subset \Rep(Q,\alpha).
\]
If $\zeta$ is $\alpha$-generic, one has $\Rep^{\zetast}(Q,\alpha) = \Rep^{\zetass}(Q,\alpha)$.

%%%%%%%%%%%%%%%%%%%%%%%%%%%%%%%%%
\subsection{Quivers with potential}
\label{subsec:quiver_with_potential}
Let $Q$ be a quiver. Consider the quotient $\BC Q / [\BC Q,\BC Q]$ of the path algebra by the vector space spanned by commutators. An element $W \in \BC Q / [\BC Q,\BC Q]$, which can be represented by a finite linear combination,
is called a \emph{potential}. Given a cyclic path $w$ and an arrow $a \in Q_1$, one defines the noncommutative derivative
\[
\frac{\partial w}{\partial a} = \sum_{\substack{w=cac' \\ c,c'\textrm{ paths in }Q}}c'c\,\in\,\BC Q.
\]
This rule extends to an operator $\partial/\partial a \colon \BC Q / [\BC Q,\BC Q] \to \BC Q$ acting on every potential.
Thus every potential  $W$ gives rise to a (two-sided) ideal $I_W \subset \BC Q$ generated by the paths $\partial W/\partial a$ for all $a \in Q_1$.
The quotient $J = J(Q,W) = \BC Q / I_{W}$
is called the \emph{Jacobi algebra} of the quiver with potential $(Q,W)$. 
For every $\alpha \in \BN^{Q_0}$, a potential $W = \sum_ca_cc$ determines a regular function
\[
f_\alpha \colon \Rep(Q,\alpha) \to \BA^1,\quad \rho \mapsto \sum_{c\textrm{ cycle in }Q}a_c \cdot \Tr (\rho(c)).
\]
The points in the critical locus $\crit (f_\alpha) \subset \Rep(Q,\alpha)$ correspond to $J$-\emph{modules} with dimension vector $\alpha$. Fix an $\alpha$-generic stability parameter $\zeta \in \BR^{Q_0}$. If $f_{\zeta,\alpha} \colon \Rep^{\zetast}(Q,\alpha)\to \BA^1$ is the restriction of $f_\alpha$, then
\[
\mathfrak M(J,\alpha) = \bigl[\crit (f_\alpha)/G_\alpha\bigr],\quad \mathfrak M_\zeta(J,\alpha) = \bigl[\crit (f_{\zeta,\alpha}) / \GL_\alpha\bigr]
\]
are, by definition, the stacks of $\alpha$-dimensional $J$-modules and $\zeta$-stable $J$-modules.  
\begin{definition}
A quiver with potential $(Q,W)$ admits a \emph{cut} if there is a subset $I \subset Q_1$ such that every cyclic monomial appearing in $W$ contains exactly one edge in $I$.
\end{definition}

From now on we assume $(Q,W)$ admits a cut. This condition ensures that the motive $[\mathfrak M(J,\alpha)]_{\vir}$ introduced in the next definition is monodromy-free, i.e.~it lives in $\widetilde{\mathcal M}_{\BC}$. See \cite[\S\,1.4]{RefConifold} for more details. All quivers considered in this paper admit a cut \cite[\S\,4]{MorrNagao}.

\begin{definition}[\cite{RefConifold}]
We define motivic Donaldson--Thomas invariants
\begin{equation}
    \label{eqn:motivicDT}
\begin{split}
    \bigl[\mathfrak M(J,\alpha)\bigr]_{\vir} 
    \,&=\, \frac{\left[\crit (f_\alpha)\right]_{\vir}}{\left[\GL_\alpha\right]_{\vir}} \\ 
    \left[\mathfrak M_\zeta(J,\alpha)\right]_{\vir} 
    \,&=\, (-\BL^{\frac{1}{2}})^{\chi(\alpha,\alpha)}\frac{\bigl[f_{\zeta,\alpha}^{-1}(0)\bigr] - \bigl[f_{\zeta,\alpha}^{-1}(1)\bigr]}{\left[\GL_\alpha\right]},
\end{split}
\end{equation}
in $\widetilde{\mathcal M}_{\BC}$, where $[\GL_\alpha]_{\vir}$ is as in Equation \eqref{example:vir_motive_smooth_scheme}.
The generating function
\begin{equation}\label{definition_AU}
A_U = \sum_{\alpha \in \BN^{Q_0}}\,\bigl[\mathfrak M(J,\alpha)\bigr]_{\vir}\cdot y^\alpha \,\in\,\CT_{Q}
\end{equation}
is called the \emph{universal series} attached to $(Q,W)$.
\end{definition}

\begin{definition}[{\cite[\S\,2.4]{RefConifold}}]\label{def:generic}
A stability parameter $\zeta \in \BR^{Q_0}$ is called \emph{generic} if $\zeta\cdot \underline{\dim}\,\rho \neq 0$ for every nontrivial $\zeta$-stable $J$-module $\rho$.
\end{definition}

%%%%%%%%%%%%%%%%%%%%%%%%%%%%%%%%%%%%%%%%%%%%%%%%%
\subsection{Framed motivic DT invariants}
\label{subsec:motivicDT_quiver_potential}

Let $r\geq 1$ be an integer, $Q$ a quiver, $\widetilde Q$ its $r$-framing with respect to a vertex $0 \in Q_0$ (Definition \ref{def:r-framing}). 
A representation $\widetilde \rho$ of $\widetilde Q$ can be uniquely written as a pair $(\rho,u)$, where $\rho$ is a representation of $Q$ and $u = (u_1,\dots,u_r)$ is an $r$-tuple of linear maps $u_i\colon \widetilde{\rho}_\infty\to \rho_0$. 

From now on, we assume all $r$-framed representations to satisfy $\dim_{\BC} \widetilde{\rho}_\infty = 1$, so that by Convention \ref{convention:dimvec} one has $\underline{\dim}\,\widetilde \rho=(\underline{\dim}\,\rho,1)$.

\begin{definition}[{\cite{NagNak} and \cite[Def.~3.1]{RefConifold}}]\label{def:framedstability}
Let $\zeta\in \mathbb R^{Q_0}$ be a stability parameter.
A representation $(\rho,u)$ of $\widetilde Q$ (or a $\widetilde J$-module) with $\dim_{\BC} \widetilde{\rho}_\infty = 1$ is said to be \emph{$\zeta$-(semi)stable} if it is $(\zeta,\zeta_\infty)$-(semi)stable in the sense of Definition \ref{stablereps}, where $\zeta_\infty = -\zeta\cdot \underline{\dim}\,\rho$.
\end{definition}

Now fix a potential $W$ on $Q$. We define motivic DT invariants for moduli stacks of $r$-framed $J$-modules on $Q$. Let $\widetilde J$ be the Jacobi algebra $J_{\widetilde{Q},W}$, where $W$ is viewed as a potential on $\widetilde Q$ in the obvious way. For a generic stability parameter $\zeta \in \BR^{Q_0}$, and a dimension vector $\alpha \in \BN^{Q_0}$, set
\[
\zeta_\infty = -\zeta\cdot \alpha,\quad \widetilde \zeta = (\zeta,\zeta_\infty),\quad \widetilde \alpha = (\alpha,1).
\]
As in \S\,\ref{subsec:quiver_with_potential}, consider the functions
\[
\begin{tikzcd}
\Rep^{\widetilde{\zeta}\textrm{-st}}(\widetilde{Q},\widetilde{\alpha})\arrow[hook]{r}\arrow[swap]{dr}{f_{\widetilde{\zeta},\widetilde{\alpha}}}
& \Rep(\widetilde{Q},\widetilde{\alpha}) \arrow{d}{f_{\widetilde{\alpha}}} \\
& \BA^1
\end{tikzcd}
\]
associated to the potential $W$. Define the moduli stacks
\[
\mathfrak M(\widetilde J,\alpha) = \bigl[\crit (f_{\widetilde{\alpha}}) \,\big/ \GL_\alpha\bigr],
\quad \mathfrak M_\zeta (\widetilde J,\alpha) = \bigl[\crit (f_{\widetilde{\zeta},\widetilde{\alpha}}) \,\big/ \GL_\alpha\bigr].
\]

\begin{definition}\label{def:motivic_partition_functions}
We define $r$-framed motivic Donaldson--Thomas invariants
\begin{align*}
    \left[\mathfrak M(\widetilde J,\alpha) \right]_{\vir}
    \,&=\,\frac{\bigl[\crit (f_{\widetilde\alpha})\bigr]_{\vir}}{\left[\GL_{\alpha}\right]_{\vir}} \\
    \left[\mathfrak M_\zeta(\widetilde J,\alpha) \right]_{\vir}
    \,&=\,\frac{\bigl[\crit (f_{\widetilde \zeta,\widetilde\alpha})\bigr]_{\vir}}{\left[\GL_{\alpha}\right]_{\vir}}
\end{align*}
in $\widetilde{\mathcal M}_{\BC}$, and the associated motivic generating functions
\begin{align*}
\widetilde A_U 
&= \sum_{\alpha \in \BN^{Q_0}} \,\left[\mathfrak M(\widetilde J,\alpha) \right]_{\vir}\cdot y^{\widetilde\alpha} \in \mathcal T_{\widetilde Q} \\
\widetilde A_\zeta 
&= \sum_{\alpha \in \BN^{Q_0}} \,\left[\mathfrak M_\zeta(\widetilde J,\alpha) \right]_{\vir}\cdot y^{\widetilde\alpha} \in \mathcal T_{\widetilde Q} \\
\mathsf{Z}_{\zeta} 
&= \sum_{\alpha \in \BN^{Q_0}} \,\left[\mathfrak M_\zeta(\widetilde J,\alpha) \right]_{\vir}\cdot y^{\alpha} \in \mathcal T_{Q}.
\end{align*}
\end{definition}

The fact that the $r$-framed invariants live in $\widetilde{\mathcal M}_{\BC}$ (i.e.~have no monodromy) follows from \cite[Lemma 1.10]{RefConifold}. The reason is that the dimension vector $\widetilde{\alpha} = (\alpha,1)$ contains `$1$' as a component.

Our main goal is to give a formula for $\mathsf{Z}_{\zeta}$, where $\zeta$ is chosen in a PT (resp.~DT) chamber.

%%%%%%%%%%%%%%%%%%%%%%%%%%%%%
\section{Noncommutative crepant resolutions}\label{sec:NCCR}
Fix integers $N_0>0$ and $0\leq N_1\leq N_0$, and set $N=N_0+N_1$.
The cone realising the singular Calabi--Yau $3$-fold $X = \Spec\, \BC[x,y,z,w]/(xy-z^{N_0}w^{N_1})$ as a toric variety is the cone over the quadrilateral $\Gamma_{N_0,N_1}$ with vertices $(0,0)$, $(N_0,0)$, $(N_1,1)$ and $(0,1)$, which becomes a triangle when $N_1=0$. 

A \emph{partition} $\sigma$ of $\Gamma_{N_0,N_1}$ is, roughly speaking, a subdivision of the polygon $\Gamma_{N_0,N_1}$ into $N$ triangles $\set{\sigma_i}_{0\leq i\leq N-1}$ of area $1/2$. We refer the reader to \cite[\S\,1.1]{Nagao_small_NCCR} for the precise definition. We denote by $\Gamma_\sigma$ the resulting object --- see Figure \ref{fig:partition} for an example with $N_0=4$, $N_1=2$. Each internal edge $\sigma_{i,i+1}$ corresponds to a component $C_{i}$ of the exceptional curve in the resolution $Y_\sigma$ attached to $\Gamma_\sigma$, and $C_{i}$ is a $(-1,-1)$-curve (resp.~a $(-2,0)$-curve) if $\sigma_{i}\cup \sigma_{i+1}$ is a quadrilateral (resp.~a triangle). 

\begin{figure}[ht]
    \centering
    \begin{tikzpicture}
\draw[black, thick] (0,0) -- (0,1);
\draw[black, thick] (0,0) -- (4,0);
\draw[black, thick] (0,1) -- (2,1);
\draw[black, thick] (2,1) -- (4,0);
\draw[red] (1,1) -- (1,0);
\draw[red] (1,1) -- (2,0);
\draw[red] (0,1) -- (1,0);
\draw[red] (2,1) -- (3,0);
\draw[red] (1,1) -- (3,0);
    \end{tikzpicture}
\caption{A partition $\Gamma_\sigma$ of $\Gamma_{4,2}$.}    
\label{fig:partition}
\end{figure}
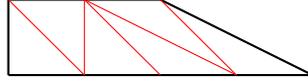

As explained in \cite{Nagao_small_NCCR, MorrNagao}, any partition $\sigma$ gives rise to a small crepant resolution $Y_\sigma \to X$ by taking the fan of $\Gamma_\sigma$, and any two such resolutions are related by a sequence of mutations. On the other hand, Nagao \cite{Nagao_small_NCCR} explains how to associate to $\sigma$ a bipartite tiling of the plane. The general construction in \cite{Hanany_Vegh} then produces a quiver with potential $(Q_\sigma,\omega_\sigma)$. Its Jacobi algebra $J_\sigma$ is derived equivalent to $Y_\sigma$ \cite[\S\,1]{Nagao_small_NCCR}. 

The quiver $Q_\sigma$ has vertex set $\widehat I = \set{0,1,\ldots,N-1}$, which we identify with the cyclic group $\BZ/N\BZ$. This in turn yields an identification 
\begin{equation}\label{identification}
\BZ^{\widehat I} = \BZ^{(Q_\sigma)_0}.
\end{equation}
Each vertex of $Q_\sigma$ has an edge in and out of the next vertex. The partition prescribes which vertices carry a loop, as we now explain using the specific example of Figure \ref{fig:partition}. In that case, the partition $\sigma=\set{\sigma_i}_{0\leq i\leq 5}$ can be identified with the ordered set of half-points 
\begin{equation}
\label{sigma}
\sigma = \Set{
\left(\frac{1}{2},0\right),
\left(\frac{1}{2},1\right),
\left(\frac{3}{2},0\right),
\left(\frac{5}{2},0\right),
\left(\frac{3}{2},1\right),
\left(\frac{7}{2},0\right)
},
\end{equation}
where the $i$th element corresponds to the mid-point of the base of the $i$th triangle $\sigma_i$.
A vertex $k \in \widehat I$ will carry a loop if and only if $\sigma_{k-1}$ and $\sigma_{k}$ have the same $y$-coordinate. Thus, by cyclicity, in our case we get two vertices ($k=0,3$) carrying a loop. The resulting quiver is drawn in Figure \ref{fig:quiver_of_partition}.

\begin{figure}[h]
\centering
\begin{tikzpicture}[>=stealth,->,shorten >=2pt,looseness=.5,auto]
  \matrix [matrix of math nodes,
           column sep={2cm,between origins},
           row sep={1.2cm,between origins},
           nodes={circle, draw, minimum size=1.5mm}]
{ 
& |(A)| 0 & |(B)| 1 & \\
|(F)| 5 & & & |(C)| 2 \\
& |(E)| 4 & |(D)| 3 & \\
};
\tikzstyle{every node}=[font=\small\itshape]
\draw (A) to [bend left=15,looseness=1] (B) node {};
\draw (B) to [bend left=15,looseness=1] (A) node {};

\draw (B) to [bend left=15,looseness=1] (C) node {};
\draw (C) to [bend left=15,looseness=1] (B) node {};

\draw (C) to [bend left=15,looseness=1] (D) node {};
\draw (D) to [bend left=15,looseness=1] (C) node {};

\draw (D) to [bend left=15,looseness=1] (E) node {};
\draw (E) to [bend left=15,looseness=1] (D) node {};

\draw (E) to [bend left=15,looseness=1] (F) node {};
\draw (F) to [bend left=15,looseness=1] (E) node {};

\draw (F) to [bend left=15,looseness=1] (A) node {};
\draw (A) to [bend left=15,looseness=1] (F) node {};

\path[->] (A) edge [loop above] node {} ();
\path[->] (D) edge [loop below] node {} ();
\end{tikzpicture}
\caption{The quiver $Q_\sigma$ associated to the partition \eqref{sigma}.}
\label{fig:quiver_of_partition}
\end{figure}
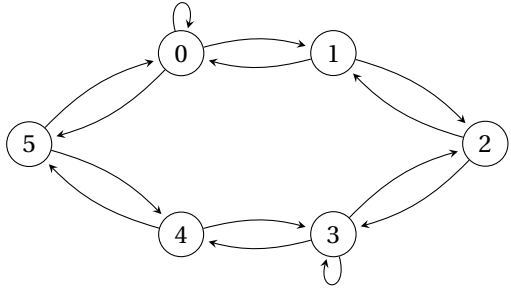

For the definition of the potential $\omega_\sigma$, we refer the reader to \cite[\S\,1.2]{Nagao_small_NCCR} or \cite[\S\,2.A]{MorrNagao}. It is proved in \cite[\S\,4]{MorrNagao} that $(Q_{\sigma},\omega_\sigma)$ has a cut for all $\sigma$.

\begin{remark}\label{Q_is_symmetric}
The quiver $Q_\sigma$ is \emph{symmetric}. This implies that its motivic quantum torus $\mathcal T_{Q_\sigma}$ is commutative.
\end{remark}

We fix $\epsilon_0,\ldots,\epsilon_{N-1}$ to be the basis of $\BZ^{(Q_\sigma)_0}$ corresponding to the canonical basis of $\BZ^{\widehat{I}}$ under \eqref{identification}. We call $\epsilon_i$ a \emph{simple root}, and $\delta = \epsilon_0 + \epsilon_1 + \cdots + \epsilon_{N-1}$ the positive minimal imaginary root. Following the notation in \cite{MorrNagao}, we set $\epsilon_{[a,b]} = \sum_{a\leq i\leq b}\epsilon_i$ for all $1\leq a\leq b\leq N-1$, and
\begin{equation}
    \label{eqn:Delta_Sets}
\begin{split}
    \Delta_{+}^{\mathrm{re},+} 
    &\,=\, \Set{\epsilon_{[a,b]} + n\cdot \delta | 1\leq a\leq b\leq N-1,\, n \in \BZ_{\geq 0}} \\
    \Delta_{+}^{\mathrm{re},-} 
    &\,=\, \Set{-\epsilon_{[a,b]} + n\cdot \delta | 1\leq a\leq b\leq N-1,\, n \in \BZ_{>0}} \\ 
    \Delta_{+}^{\mathrm{im}} 
    &\,=\,\Set{n \cdot \delta | n \in \BZ_{>0}}.
\end{split}
\end{equation}
From the above sets we form the larger sets
\[
\Delta_+^{\mathrm{re}} \,=\, \Delta_{+}^{\mathrm{re},+}  \amalg \Delta_{+}^{\mathrm{re},-}, \quad \Delta_+ \,=\,\Delta_+^{\mathrm{re}} \amalg \Delta_{+}^{\mathrm{im}}.
\]

\begin{remark}
The above sets depend on $\sigma$, but we omit this dependence to ease notation; in the language of \cite{MorrNagao}, we have $\Delta_+ = \Delta_{\sigma,+}$, $\Delta_+^{\mathrm{re}} = \Delta_{\sigma,+}^{\mathrm{re}}$ and $\Delta_+^{\mathrm{im}} = \Delta_{\sigma,+}^{\mathrm{im}}$.
\end{remark}

%%%%%%%%%%%%%%%%%%%%%%%%%%%%%%%%%
\section{Higher rank motivic DT theory of points}\label{sec:motive_quot}
The rank $1$ DT theory of points on a $3$-fold $Y$ is entirely solved, see e.g.~\cite{BFHilb} for the case of $\Hilb^nY$ and \cite{RGKummer} for the \emph{reduced} DT theory of points on an abelian $3$-fold. In higher rank, to define the theory we fix a locally free sheaf $F$ of rank $r$ on $Y$. Building on the case of $Y=\BA^3$, fully explored in \cite{Cazzaniga_Thesis, ThesisR, refinedDT_asymptotics,cazzaniga2020framed}, a virtual motive for the Quot scheme $\Quot_Y(F,n)$ was defined in \cite[Def.~4.10]{Quot19} via power structures, along the same lines of the rank $1$ case \cite[\S\,4.1]{BBS}.

The generating function
\[
\DT_r^{\points}(Y,(-1)^rs) = \sum_{n\geq 0}\,\bigl[\Quot_Y(F,n) \bigr]_{\vir}\cdot ((-1)^rs)^n
\]
was computed in \cite[Thm.~4.11]{Quot19} as a plethystic exponential. Just as in the case of the naive motives \cite{ricolfi2019motive}, the generating function does not depend on $F$ but only on $r$ and on the motive of $Y$.

Consider the singular affine toric Calabi--Yau $3$-fold $X =  \Spec \BC[x,y,z,w]/(xy-z^{N_0}w^{N_1}) \subset\BA^4$,
and fix a partition $\sigma$ associated to the polygon $\Gamma_{N_0,N_1}$.

\begin{lemma}\label{lemma:motive_of_Ysigma}
Let $Y_\sigma$ be the crepant resolution of $X$ corresponding to $\sigma$. Then
\[
[Y_\sigma] = \BL^3+(N-1)\BL^2\,\in\,K_0(\Var_{\BC}).
\]
\end{lemma}

\begin{proof}
The toric polygon of $Y_\sigma$ consists of $N=N_{0}+N_{1}$ triangles $\{\sigma_i\}$ intersecting pairwise along the edges $\{\sigma_{i, i+1}\}$.  The toric resolution $Y_\sigma$ is constructed by gluing the toric charts $U_{\sigma_i}$ along the open affine subvarieties $U_{\sigma_{i, i+1}}$. Thus, the class $[Y_\sigma]$ can be computed using the cut-and-paste relations, after noticing that $U_{\sigma_i}\simeq \BA^{3}$ and $U_{\sigma_{i,i+1}}\simeq \BA^{2}\times \BC^{\ast}$. The result is
\[
[Y_\sigma] = \sum_{i=1}^{N} \BL^{3} - \sum_{i=1}^{N-1} \BL^{2}(\BL-1)= \BL^{3}+ (N-1) \BL^{2}. \qedhere
\]
\end{proof}

By \cite[Thm.~A]{refinedDT_asymptotics} (but see also \cite{Cazzaniga_Thesis,ThesisR} for different proofs), after rephrasing the result using the conventions adopted in this paper (cf.~Remark \ref{remark:different_conventions}), one has
\[
\DT_r^{\points}(\BA^3,(-1)^rs) 
\,=\, \prod_{m\geq 1}\prod_{k=0}^{rm-1}\left(1-\BL^{k+2-\frac{rm}{2}}s^m\right)^{-1}
\,=\,\prod_{i=1}^r\DT_1^{\points}\left(\BA^3,-\BL^{\frac{-r-1}{2}+i}s\right).
\]
An easy power structure argument shows that the same decomposition into $r$ rank $1$ pieces holds for every smooth $3$-fold $Y$. In a little more detail (we refer the reader to \cite{GLMps} or to \cite{BBS,DavisonR} for the formal properties of the power structure on $\CM_{\BC}$), we have
\begin{align*}
    \DT_r^{\points}(Y,(-1)^rs)
    &\,=\, \DT_r^{\points}\left(\BA^3,(-1)^rs\right)^{\BL^{-3}[Y]} \\
    &\,=\, \prod_{i=1}^r\DT_1^{\points}\left(\BA^3,-\BL^{\frac{-r-1}{2}+i}s\right)^{\BL^{-3}[Y]} \\
    &\,=\,\prod_{i=1}^r\DT_1^{\points}\left(Y,-\BL^{\frac{-r-1}{2}+i}s\right).
\end{align*}
Therefore, for any smooth $3$-fold $Y$, we can write
\begin{equation}\label{eqn:DT_r_points_NCCR}
\DT_r^{\points}(Y,s) = \prod_{i=1}^r\DT_1^{\points}\left(Y,(-1)^{r+1}\BL^{\frac{-r-1}{2}+i}s\right).
\end{equation}
By Lemma \ref{lemma:motive_of_Ysigma}, the motivic partition of the Hilbert scheme of points on $Y_\sigma$ is 
\begin{equation}\label{eqn:Hilb_points_Ysigma}
\DT_1^{\points}(Y_\sigma,s)=
\prod_{m\geq 1}\prod_{k=0}^{m-1}\left(1-\BL^{k+1-\frac{m}{2}}(-s)^m\right)^{1-N}\left(1-\BL^{k+2-\frac{m}{2}}(-s)^m\right)^{-1}
\end{equation}
and this determines $\DT_r^{\points}(Y_\sigma,s)$ via Equation \eqref{eqn:DT_r_points_NCCR}. The result is
\begin{equation}\label{DT_points_Y_sigma}
\DT_r^{\points}(Y_\sigma,s) = \prod_{m\geq 1}\prod_{k=0}^{rm-1}\left(1-\BL^{k+1-\frac{rm}{2}}((-1)^rs)^m\right)^{1-N}\left(1-\BL^{k+2-\frac{rm}{2}}((-1)^rs)^m\right)^{-1}.
\end{equation}

%%%%%%%%%%%%%%%%%%%%%%%%%%%%%%%%%%%%%%%%%%%%%%%%%%%%%%%%%%%%%%%%%%
\section{Motivic invariants of noncommutative crepant resolutions}

\subsection{Relations among motivic partition functions}

Fix integers $N_0>0$ and $0\leq N_1\leq N_0$, and set $N=N_0+N_1$.
We consider the affine singular toric Calabi--Yau $3$-fold
\[
X_{N_0,N_1} = \Spec \BC[x,y,z,w]/(xy-z^{N_0}w^{N_1}) \,\subset\,\BA^4.
\]
Fix a partition $\sigma$ of the polygon $\Gamma_{N_0,N_1}$, and set $(Q,W,J) = (Q_\sigma,\omega_\sigma,J_\sigma)$ to ease notation, where $J_\sigma$ is the Jacobi algebra of the quiver with potential $(Q_\sigma,\omega_\sigma)$ whose construction we sketched in \S\,\ref{sec:NCCR}.
The universal series
\[
A_U^{\sigma}(y) = A_U^{\sigma}(y_0,\ldots,y_{N-1}) = \sum_{\alpha \in \BN^{Q_0}} \,\bigl[\mathfrak{M}(J_\sigma,\alpha) \bigr]_{\vir} \cdot y^\alpha \,\in\,\mathcal T_{Q},
\]
defined in Equation \eqref{definition_AU},
is the main object of study in the work of Morrison and Nagao \cite{MorrNagao}.

Fix a generic stability parameter $\zeta$ (cf.~Definition \ref{def:generic}) on the unframed quiver $Q$. Consider the stacks $\mathfrak M^{\pm}_\zeta(J,\alpha)$ of $J$-modules all of whose Harder--Narasimhan factors have positive (resp.~negative) slope with respect to $\zeta$. These stacks are defined as follows. Restrict the function $f_\alpha \colon \Rep(Q,\alpha) \to \BA^1$, defined by taking the trace of $\omega_\sigma$, to the open subschemes $\Rep^{\pm}_{\zeta}(Q,\alpha) \subset \Rep(Q,\alpha)$ of representations satisfying the above properties. This yields two regular functions $f_\zeta^\pm \colon \Rep^{\pm}_{\zeta}(Q,\alpha) \to \BA^1$, and we set $\mathfrak M^{\pm}_\zeta(J,\alpha) = [\crit (f_\zeta^\pm) / \GL_\alpha]$. We define the virtual motives $[\mathfrak M^{\pm}_\zeta(J,\alpha)]_{\vir}$ as in the second identity in Equation \eqref{eqn:motivicDT}, and the associated motivic generating functions (depending on $\sigma$ via $J=J_\sigma$)
\[
A_\zeta^\pm = \sum_{\alpha \in \BN^{Q_0}}\,\left[\mathfrak M^{\pm}_\zeta(J,\alpha) \right]_{\vir}\cdot y^\alpha \in \mathcal T_Q.
\]
The vertices of $Q$ are labeled from $0$ up to $N-1$. Let $\widetilde Q$ be the $r$-framed quiver associated to $(Q,0)$ (Definition \ref{def:r-framing}). We let $\widetilde J = J_{\widetilde{Q},W}$ be the Jacobi algebra of $(\widetilde{Q},W) = (\widetilde Q_\sigma,\omega_\sigma)$.
Now recall the motivic generating functions
\[
\widetilde{A}_U,\quad \widetilde A_\zeta,\quad \mathsf{Z}_{\zeta}
\]
introduced in Definition \ref{def:motivic_partition_functions}.
We have to  extend the relations between framed and unframed generating functions (in the same spirit of Mozgovoy's work \cite{Mozgovoy_Framed_WC}) to general $r$. By the following lemma, the arguments are going to be essentially formal.

\begin{lemma}\label{lemma:switch_alpha_infinity}
In $\mathcal T_{\widetilde{Q}}$ there are identities
\[
y_\infty \cdot y^{(\alpha,0)} =  (-\BL^{\frac{1}{2}})^{-r\alpha_0}\cdot y^{\widetilde{\alpha}}, \quad y^{(\alpha,0)} \cdot y_\infty = (-\BL^{\frac{1}{2}})^{r\alpha_0}\cdot y^{\widetilde{\alpha}}.
\]
\end{lemma}

\begin{proof}
Since $\infty \in \widetilde{Q}_0$ has edges only reaching $0$, and no vertex of $Q$ reaches $\infty$, we have $\chi((\alpha,0),(\mathbf 0,1)) = 0$, and $\chi((\mathbf 0,1),(\alpha,0))=-r\alpha_0$. The result follows by the product rule \eqref{eqn:product_in_TQ}.
\end{proof}

\begin{corollary}\label{dec1}
In $\mathcal T_{\widetilde{Q}}$, there are identities
\begin{align}
    \widetilde A_\zeta 
    &= y_\infty\cdot \mathsf{Z}_{\zeta}((-\BL^{\frac{1}{2}})^ry_0,y_1,\ldots,y_{N-1})\label{eqnn1}\\
    A_\zeta^-\cdot y_\infty 
    &= y_\infty\cdot A_\zeta^-(\BL^ry_0,y_1,\ldots,y_{N-1}).\label{eqnn2}
\end{align}
\end{corollary}

\begin{proof}
We have 
\begin{align*}
    y_\infty\cdot \mathsf{Z}_{\zeta}((-\BL^{\frac{1}{2}})^ry_0,y_1,\ldots,y_{N-1}) 
    &= \sum_{\alpha\in\BN^{Q_0}} \,\left[\mathfrak M_\zeta(\widetilde J,\alpha) \right]_{\vir} \cdot y_\infty\cdot ((-\BL^{\frac{1}{2}})^r y_0)^{\alpha_0}\cdot y_1^{\alpha_1}\cdots y_{N-1}^{\alpha_{N-1}} \\
    &= \sum_{\alpha\in\BN^{Q_0}} \,\left[\mathfrak M_\zeta(\widetilde J,\alpha) \right]_{\vir}  (-\BL^{\frac{1}{2}})^{r\alpha_0}\cdot (y_\infty\cdot  y^{\alpha}) \\
    &= \sum_{\alpha\in\BN^{Q_0}} \,\left[\mathfrak M_\zeta(\widetilde J,\alpha) \right]_{\vir}\cdot y^{\widetilde{\alpha}} \\
    &= \widetilde A_\zeta,
\end{align*}
where we have applied Lemma \ref{lemma:switch_alpha_infinity} in the last step. The identity \eqref{eqnn2} follows by an identical argument.
\end{proof}

\begin{lemma}[{\cite[Proposition 3.5]{RefConifold}}]\label{prop:filtration}
Let $Q$ be a quiver, $\zeta \in \BR^{Q_0}$ a generic stability parameter, $\widetilde{\rho}$ a representation (resp.~$\widetilde{J}$-module) of the $r$-framed quiver $\widetilde{Q}$ with $\dim_{\BC}\widetilde{\rho}_\infty = 1$. Then there is a unique filtration $0=\widetilde{\rho}^{0}\subset \widetilde{\rho}^{1}\subset \widetilde{\rho}^{2}\subset \widetilde{\rho}^{3}=\widetilde{\rho}$ such that the quotients $\widetilde{\pi}^{i}=\widetilde{\rho}^{i}/\widetilde{\rho}^{i-1}$ satisfy:
\begin{enumerate}
    \item $\widetilde{\pi}^{1}_{\infty}=0$, and $\widetilde{\pi}^{1}\in \Rep_{\zeta}^{+}(Q,\underline{\dim}\,\widetilde{\pi}^{1})$,
    \item $\dim_{\BC}\widetilde{\pi}^{2}_{\infty}=1$ and $\widetilde{\pi}^{2}$ is $\zeta$-stable,
    \item $\widetilde{\pi}^{3}_{\infty}=0$, and $\widetilde{\pi}^{3}\in \Rep_{\zeta}^{-}(Q,\underline{\dim}\,\widetilde{\pi}^{3})$.
\end{enumerate}
\end{lemma}

\begin{lemma}\label{dec2}
Let $\zeta \in \BR^{Q_0}$ be a generic stability parameter. In $\mathcal T_{\widetilde{Q}}$, there are  factorisations
\begin{align}
    \widetilde A_U &= A^+_\zeta\cdot \widetilde A_\zeta \cdot A^-_\zeta \label{eqn1} \\
    \widetilde A_U &= A_U^\sigma \cdot y_\infty. \label{eqn2}
\end{align}
\end{lemma}

\begin{proof}
Equation \eqref{eqn1} is a direct consequence of the existence of the filtration of Lemma \ref{prop:filtration}. Equation \eqref{eqn2} follows directly from the following observation: given a framed representation $\widetilde{\rho} = (\rho,u)$ with $\dim_{\BC}\widetilde{\rho}_\infty = 1$, one can view $\rho$ as a sub-module $\rho\subset\widetilde{\rho}$ of dimension $(\underline{\dim}\,\rho,0)$, and the quotient $\widetilde{\rho}/\rho$ is the unique simple module of dimension $(0,1)$, based at the framing vertex.
\end{proof}

Following \cite[\S\,0]{MorrNagao}, we define, for $\alpha \in \Delta_{+}$, the infinite products
\begin{equation}\label{eqn:A_alpha}
    A_\alpha(y) = 
    \begin{cases}
    \displaystyle\prod_{j\geq 0}\left(1-\BL^{-j-\frac{1}{2}}y^\alpha\right) 
    & \textrm{if }\alpha \in \Delta^{\mathrm{re}}_{+}\textrm{ and } \sum_{k \notin \widehat{I}_\ell}\alpha_k\textrm{ is odd}\\ \\
    \displaystyle\prod_{j\geq 0}\left(1-\BL^{-j}y^\alpha\right)^{-1} & \textrm{if }\alpha \in \Delta^{\mathrm{re}}_{+}\textrm{ and } \sum_{k \notin \widehat{I}_\ell}\alpha_k\textrm{ is even}\\ \\
    \displaystyle\prod_{j\geq 0}\left(1-\BL^{-j}y^\alpha\right)^{1-N}\left(1-\BL^{-j+1}y^\alpha\right)^{-1} & \textrm{if }\alpha \in \Delta^{\mathrm{im}}_{+}
    \end{cases}
\end{equation}
where $\widehat{I}_\ell \subset \widehat I = (Q_\sigma)_0$ denotes\footnote{The set $\widehat{I}_\ell$ is denoted $\widehat{I}_r$ in \cite{MorrNagao}. We changed the notation to avoid conflict with the number $r$ of framings.} the set of vertices carrying a loop, and $\alpha_k \in \BN$ is the component of $\alpha$ corresponding to a vertex $k$. 

\begin{lemma}[{\cite[Lemma 2.6]{RefConifold}}]
Let $\zeta \in \BR^{Q_0}$ be a generic stability parameter. In $\mathcal T_Q$, there are identities
\begin{equation}\label{eqn:Apm_factorisation}
A^{\pm}_\zeta(y) = \prod_{\substack{\alpha \in \Delta_{+} \\ \pm\zeta\cdot \alpha > 0}} A_\alpha(y).
\end{equation}
\end{lemma}

\begin{lemma}\label{lemma:uni=plus_times_minus}
Let $\zeta \in \BR^{Q_0}$ be a generic stability parameter. In $\mathcal T_{Q}$, there is an identity
\begin{equation}\label{eqn:plus_times_minus}
A_{U}^\sigma = A_\zeta^+ \cdot A_\zeta^-.
\end{equation}
\end{lemma}

\begin{proof}
By \cite[Thm.~0.1]{MorrNagao} there is a factorisation
\[
A_U^\sigma(y) = \prod_{\alpha \in \Delta_{+}} A_\alpha(y).
\]
Since $\zeta$ is generic, $\zeta\cdot \alpha \neq 0$ for all $\alpha \in \Delta_{+}$. The result then follows by combining this factorisation with Equation \eqref{eqn:Apm_factorisation}.
\end{proof}

\begin{theorem}
Let $\zeta \in \BR^{Q_0}$ be a generic stability parameter. In $\mathcal T_Q$, there is an identity
\begin{equation}\label{eqn:important_fraction}
\mathsf{Z}_{\zeta}(y) = \frac{A_\zeta^-((-\BL^{\frac{1}{2}})^ry_0,y_1,\ldots,y_{N-1})}{A_\zeta^-((-\BL^{-\frac{1}{2}})^ry_0,y_1,\ldots,y_{N-1})}.
\end{equation}
\end{theorem}

\begin{proof}
Since $Q = Q_\sigma$ is symmetric (Remark \ref{Q_is_symmetric}), the algebra $\mathcal T_{Q}$ is commutative, therefore a power series $F \in \mathcal T_{Q}$ starting with the invertible element $1 \in \widetilde{\mathcal M}_{\BC}$ will be invertible. For instance $A_\zeta^+$ and $A_\zeta^-$ are invertible. Therefore we can write
\begin{align*}
    y_\infty\cdot \mathsf{Z}_{\zeta}((-\BL^{\frac{1}{2}})^ry_0,y_1,\ldots,y_{N-1}) &= \widetilde{A}_\zeta & \textrm{by }\eqref{eqnn1} \\
    &= (A_\zeta^+)^{-1}\cdot \widetilde{A}_U\cdot (A_\zeta^-)^{-1} & \textrm{by }\eqref{eqn1} \\
    &= (A_\zeta^+)^{-1}\cdot (A_U^\sigma\cdot y_\infty)\cdot (A_\zeta^-)^{-1} & \textrm{by }\eqref{eqn2} \\
    &= (A_\zeta^+)^{-1}\cdot (A_\zeta^+ \cdot A_\zeta^-\cdot y_\infty) \cdot (A_\zeta^-)^{-1} & \textrm{by }\eqref{eqn:plus_times_minus} \\
    &= y_\infty \cdot A_\zeta^-(\BL^ry_0,y_1,\ldots,y_{N-1})\cdot (A_\zeta^-)^{-1} & \textrm{by }\eqref{eqnn2}
\end{align*}
from which it follows that
\[
\mathsf{Z}_{\zeta}((-\BL^{\frac{1}{2}})^ry_0,y_1,\ldots,y_{N-1}) = \frac{A_\zeta^-(\BL^ry_0,y_1,\ldots,y_{N-1})}{A_\zeta^-(y_0,y_1,\ldots,y_{N-1})}.
\]
Thus the change of variable $y_0 \to (-\BL^{-\frac{1}{2}})^ry_0$ yields the result.
\end{proof}

\subsection{Computing invariants in the DT and PT chambers}\label{subsec:Computing_invariants}

In this subsection we prove Theorem \ref{mainthm}.

Define, for $\alpha \in \Delta_{+}$, the fraction
\begin{equation}\label{eqn:definition_Zeta^r}
Z_\alpha^{(r)}(y_0,y_1,\ldots,y_{N-1}) = \frac{A_\alpha((-\BL^{\frac{1}{2}})^ry_0,y_1,\ldots,y_{N-1})}{A_\alpha((-\BL^{-\frac{1}{2}})^ry_0,y_1,\ldots,y_{N-1})},
\end{equation}
where $A_\alpha$ is defined case by case in \eqref{eqn:A_alpha}. Then one deduces the following explicit formulae:
\[
    Z_\alpha^{(r)}((-1)^ry_0,y_1,\ldots,y_{N-1}) = 
    \begin{cases}
    \displaystyle\prod_{k=0}^{r\alpha_0-1}\left(1-\BL^{k+\frac{1}{2}-\frac{r\alpha_0}{2}}y^\alpha\right)
     & \textrm{if }\alpha \in \Delta^{\mathrm{re}}_{+}\textrm{ and } \sum_{k \notin \widehat{I}_\ell}\alpha_k\textrm{ is odd}\\ \\
     \displaystyle\prod_{k=0}^{r\alpha_0-1}\left(1-\BL^{k+1-\frac{r\alpha_0}{2}}y^\alpha\right)^{-1}
     & \textrm{if }\alpha \in \Delta^{\mathrm{re}}_{+}\textrm{ and } \sum_{k \notin \widehat{I}_\ell}\alpha_k\textrm{ is even}\\ \\
     \displaystyle\prod_{k=0}^{r\alpha_0-1}\left(1-\BL^{k+1-\frac{r\alpha_0}{2}}y^\alpha\right)^{1-N}\left(1-\mathbb L^{k+2-\frac{rm}{2}}y^\alpha \right)^{-1}
     & \textrm{if }\alpha \in \Delta^{\mathrm{im}}_{+}.
    \end{cases}
\]
These identities can be easily rewritten uniformly in terms of the `rank $1$' generating functions:
\begin{equation}\label{eqn:rank_r_as_a_product}
    Z_\alpha^{(r)}((-1)^ry_0,y_1,\ldots,y_{N-1}) = \prod_{i=1}^r Z_\alpha^{(1)}\left(-\BL^{\frac{-r-1}{2}+i}y_0,y_1,\ldots,y_{N-1}\right).
\end{equation}

Let us set
\[
s = y_0y_1\cdots y_{N-1},\quad T_i = y_i^{-1}, \quad T = (T_1,\ldots,T_{N-1}).
\]
For $1\leq a\leq b\leq N-1$, we let $T_{[a,b]} = T_a\cdots T_b$ be the monomial corresponding to the homology class $C_{[a,b]} = [C_a] + \cdots + [C_b] \in H_2(Y_\sigma,\BZ)$, where $C_i\subset Y_\sigma$ is a component of the exceptional curve. Let $c(a,b)$ be the number of $(-1,-1)$-curves in $\set{C_i|a\leq i\leq b}$. Then we set
\[
Z_{[a,b]}(s,T_{[a,b]}) = 
\begin{cases}
\displaystyle\prod_{m\geq 1}\prod_{j=0}^{m-1}\left(1-\BL^{j+\frac{1}{2}-\frac{m}{2}}(-s)^mT_{[a,b]} \right) & \textrm{if }c(a,b)\textrm{ is odd} \\ \\
\displaystyle\prod_{m\geq 1}\prod_{j=0}^{m-1}\left(1-\BL^{j+1-\frac{m}{2}}(-s)^mT_{[a,b]} \right)^{-1} & \textrm{if }c(a,b)\textrm{ is even}
\end{cases}
\]
and
\[
Z_{\mathrm{im}}(s) = \prod_{m\geq 1}\prod_{j=0}^{m-1} \left(1-\BL^{j+1-\frac{m}{2}}(-s)^m\right)^{1-N}\left(1-\BL^{j+2-\frac{m}{2}}(-s)^m\right)^{-1}.
\]
Fix, as in \cite[\S\,6.C]{MorrNagao}, stability parameters
\[
\zeta_{\PT} = (1-N+\varepsilon,1,\ldots,1),\quad \zeta_{\DT} = (1-N-\varepsilon,1,\ldots,1),
\]
with $0 < \varepsilon \ll 1$ chosen so that they are generic. We want to compute
\[
\PT_r(Y_\sigma;s,T) = \mathsf Z_{\zeta_{\PT}}(s,T_1,\ldots,T_{N-1}),\quad \DT_r(Y_\sigma;s,T) = \mathsf Z_{\zeta_{\DT}}(s,T_1,\ldots,T_{N-1}).
\]
For $r=1$, these are the generating functions computed in \cite[Cor.~0.3]{MorrNagao}.
We know by Equation \eqref{eqn:Hilb_points_Ysigma} (see also \cite[Cor.~0.3\,(2)]{MorrNagao}) that 
\begin{equation}\label{eqn:Z_im=DT_1}
Z_{\mathrm{im}}(s) = \DT^{\points}_1(Y_\sigma,s),
\end{equation}
and Morrison--Nagao proved that 
\begin{equation}
\label{eqn:rank_one_PT}
\begin{split}
\PT_1(Y_\sigma;s,T) &= \prod_{1\leq a\leq b\leq N-1}Z_{[a,b]}(s,T_{[a,b]}) \\
\DT_1(Y_\sigma;s,T) &= Z_{\mathrm{im}}(s)\cdot \PT_1(Y_\sigma;s,T).
\end{split}
\end{equation}

We have
\begin{equation}\label{eqn:DT/PT_range}
    \begin{split}
        \Set{\alpha \in \Delta_{+}|\zeta_{\PT}\cdot \alpha < 0} &= \Delta_{+}^{\mathrm{re},-} \\
        \Set{\alpha \in \Delta_{+}|\zeta_{\DT}\cdot \alpha < 0} &= \Delta_{+}^{\mathrm{re},-} \amalg \Delta_{+}^{\mathrm{im}},
    \end{split}
\end{equation}
where the definition of the sets in the right hand sides was given in Equation \eqref{eqn:Delta_Sets}. For the PT stability condition, we thus obtain
\begin{align*}
    \PT_r(Y_\sigma;s,T) 
    &= \frac{A^-_{\zeta_{\PT}}((-\BL^{\frac{1}{2}})^rs,T_1,\ldots,T_{N-1} )}{A^-_{\zeta_{\PT}}((-\BL^{-\frac{1}{2}})^rs,T_1,\ldots,T_{N-1})} & \textrm{by }\eqref{eqn:important_fraction} \\
    &=\prod_{\alpha \in \Delta_{+}^{\mathrm{re},-}} \frac{A_\alpha((-\BL^{\frac{1}{2}})^rs,T_1,\ldots,T_{N-1} )}{A_\alpha((-\BL^{-\frac{1}{2}})^rs,T_1,\ldots,T_{N-1} )} & \textrm{by }\eqref{eqn:Apm_factorisation}\textrm{ and }\eqref{eqn:DT/PT_range} \\
     &= \prod_{\alpha \in \Delta_{+}^{\mathrm{re},-}} Z_\alpha^{(r)}\left(s,T_1,\ldots,T_{N-1}\right) & \textrm{by }\eqref{eqn:definition_Zeta^r} \\
     &=\prod_{i=1}^r \prod_{\alpha \in \Delta_{+}^{\mathrm{re},-}} Z_\alpha^{(1)}\left((-1)^{r+1}\BL^{\frac{-r-1}{2}+i}s,T_1,\ldots,T_{N-1}\right) & \textrm{by }\eqref{eqn:rank_r_as_a_product} \\
     &=\prod_{i=1}^r\prod_{1\leq a\leq b\leq N-1} Z_{[a,b]}\left((-1)^{r+1}\BL^{\frac{-r-1}{2}+i}s,T_{[a,b]}\right) & \textrm{by }\eqref{eqn:Delta_Sets}\\
     &=\prod_{i=1}^r\PT_1\left(Y_\sigma;(-1)^{r+1}\BL^{\frac{-r-1}{2}+i}s,T\right), & \textrm{by }\eqref{eqn:rank_one_PT}
\end{align*}
which proves the first identity in Theorem \ref{mainthm}.

Similarly,
\begin{align*}
\prod_{\alpha \in \Delta_+^{\mathrm{im}}} \frac{A_\alpha((-\BL^{\frac{1}{2}})^rs,T_1,\ldots,T_{N-1} )}{A_\alpha((-\BL^{-\frac{1}{2}})^rs,T_1,\ldots,T_{N-1})} 
&= \prod_{\alpha \in \Delta_+^{\mathrm{im}}} Z_\alpha^{(r)}(s,T_1,\ldots,T_{N-1}) & \textrm{by }\eqref{eqn:definition_Zeta^r} \\
&= \prod_{i=1}^r Z_{\mathrm{im}}((-1)^{r+1}\BL^{\frac{-r-1}{2}+i}s) & \textrm{by }\eqref{eqn:rank_r_as_a_product} \\
&= \prod_{i=1}^r \DT_1^{\points}\left(Y_\sigma,(-1)^{r+1}\BL^{\frac{-r-1}{2}+i}s\right) & \textrm{by }\eqref{eqn:Z_im=DT_1} \\
&= \DT_r^{\points}(Y_\sigma,s). & \textrm{by }\eqref{eqn:DT_r_points_NCCR}
\end{align*}
In particular, thanks to \eqref{eqn:DT/PT_range}, the motivic DT/PT correspondence
\[
\DT_r(Y_\sigma;s,T)=\DT_r^{\points}(Y_\sigma,s)\cdot\PT_r(Y_\sigma;s,T)
\]
holds. Note that, thanks to Equation \eqref{DT_points_Y_sigma}, the right hand side is entirely explicit. Finally, the relation
\[
\DT_r(Y_\sigma;s,T) = \prod_{i=1}^r\DT_1\left(Y_\sigma;(-1)^{r+1}s\BL^{\frac{-r-1}{2}+i},T\right)
\]
follows from the factorisations of $\PT_r$ and $\DT_r^{\points}$ as products of (equally shifted) $r=1$ pieces, combined with the rank $1$ DT/PT correspondence \eqref{eqn:rank_one_PT}. The proof of Theorem \ref{mainthm} is complete.

\begin{remark}
A motivic DT/PT correspondence was obtained in \cite{DavisonR} in the rank $1$ case for the motivic contribution of a smooth curve in a $3$-fold, refining the corresponding enumerative calculations \cite{LocalDT,Ricolfi2018}.
\end{remark}

\begin{remark}\label{remark:geometric_interpretation}
In the case when $Y_\sigma$ is the crepant resolution of the conifold singularity, corresponding to $N_0=N_1=1$, the moduli space of framed quiver representation has a clear geometric interpretation for a choice of PT stability condition. Consider the moduli space $\mathcal{P}^{r}_{\alpha}(Y_\sigma)$ parametrising Shesmani's highly frozen stable triples \cite{sheshmani}, whose geometric points consist of framed multi-sections $\OO_{Y_\sigma}^{\oplus r}\rightarrow F$ with $0$-dimensional cokernel, where $F$ is a pure $1$-dimensional sheaf $F$ satisfying $\mathrm{ch}_{2}(F) = (\alpha_0-\alpha_1)[\mathbb P^{1}]$ and $\chi(F) = \alpha_0$. In \cite[Chap. 3]{Cazzaniga_Thesis} a scheme theoretic isomorphism $\mathfrak{M}_{\zeta_{\PT}}(\widetilde{J}_\sigma,\alpha)\simeq \mathcal{P}^{r}_{\alpha}(Y_{\sigma})$ is constructed, and it is used to compute a first instance of Formula \eqref{eqn:main_formulae}.
A completely analogous result holds when $Y_\sigma$ is the resolution of a line of $A_2$ singularities, corresponding to the case $N_0=2, N_1=0$ \cite[Appendix 3.A]{Cazzaniga_Thesis}. We leave to future work a full geometric interpretation of the more general moduli spaces of framed quiver representations that we studied in this paper.   
\end{remark}

\subsection*{Acknowledgments}
We thank Bal\'{a}zs Szendr\H{o}i for reading a first draft of this paper and for providing us with helpful comments and suggestions. We thank the anonymous referee for suggesting various improvements and corrections. A.C.~thanks CNR-IOM for support and the excellent working conditions. A.R.~thanks Dipartimenti di Eccellenza for support and SISSA for the excellent working conditions.

\bibliographystyle{amsplain-nodash}
\bibliography{bib}

\ifx\undefined\bysame
\newcommand{\bysame}{\leavevmode\hbox to3em{\hrulefill}\,}
\fi
\begin{thebibliography}{10}

\bibitem{BR18}
Sjoerd Beentjes and Andrea~T. Ricolfi, {\em {Virtual counts on Quot schemes and
  the higher rank local DT/PT correspondence}},
  {\href{https://arxiv.org/abs/1811.09859}{ArXiv:1811.09859}}, To appear in
  Math. Res. Lett., 2018.

\bibitem{BBS}
Kai Behrend, Jim Bryan, and Bal{{\'a}}zs Szendr{\H{o}}i, {\em Motivic degree
  zero {D}onaldson--{T}homas invariants}, Invent. Math. {\bf 192} (2013),
  no.~1, 111--160.

\bibitem{Behrend_Ajneet}
Kai Behrend and Ajneet Dhillon, {\em On the motivic class of the stack of
  bundles}, Adv. Math. {\bf 212} (2007), no.~2, 617--644.

\bibitem{BFHilb}
Kai Behrend and Barbara Fantechi, {\em {\itshape{Symmetric obstruction theories
  and Hilbert schemes of points on threefolds}}}, Algebra Number Theory {\bf 2}
  (2008), 313--345.

\bibitem{Cazzaniga_Thesis}
Alberto Cazzaniga, {\em {On some computations of refined Donaldson--Thomas
  invariants}}, PhD Thesis, University of Oxford, 2015.

\bibitem{refinedDT_asymptotics}
Alberto Cazzaniga, Dimbinaina Ralaivaosaona, and Andrea~T. Ricolfi, {\em
  {Higher rank motivic Donaldson--Thomas invariants of $\mathbb A^3$ via
  wall-crossing, and asymptotics}},
  \href{https://arxiv.org/abs/2004.07020}{ArXiv:2004.07020}, 2020.

\bibitem{cazzaniga2020framed}
Alberto Cazzaniga and Andrea~T. Ricolfi, {\em {Framed sheaves on projective
  space and Quot schemes}},
  {\href{https://arxiv.org/abs/2004.13633}{ArXiv:2004.13633}}, 2020.

\bibitem{DavisonR}
Ben Davison and Andrea~T. Ricolfi, {\em {The local motivic DT/PT
  correspondence}},
  {\href{https://arxiv.org/abs/1905.12458}{ArXiv:1905.12458}}, 2019.

\bibitem{DenefLoeser1}
Jan Denef and Fran{\c c}ois Loeser, {\em {Geometry on arc spaces of algebraic
  varieties}}, {3rd European congress of mathematics (ECM), Barcelona, Spain,
  July 10--14, 2000. Volume I}, Basel: Birkh\"auser, 2001, pp.~327--348.

\bibitem{FMR_K-DT}
Nadir Fasola, Sergej Monavari, and Andrea~T. Ricolfi, {\em {Higher rank
  K-theoretic Donaldson--Thomas theory of points}}, To appear in Forum Math.
  Sigma, DOI: 10.1017/fms.2021.4, 2021.

\bibitem{RGKummer}
Martin~G. Gulbrandsen and Andrea~T. Ricolfi, {\em {The Euler charateristic of
  the generalized Kummer scheme of an abelian threefold}}, {Geom. Dedicata}
  {\bf 182} (2016), 73--79.

\bibitem{GLMps}
Sabir~M. {Gusein-Zade}, Ignacio {Luengo}, and Alejandro {Melle-Hern\'andez},
  {\em {A power structure over the Grothendieck ring of varieties}}, {Math.
  Res. Lett.} {\bf 11} (2004), no.~1, 49--57.

\bibitem{Hanany_Vegh}
Amihay Hanany and David Vegh, {\em Quivers, tilings, branes and rhombi}, J.
  High Energy Phys. {\bf 2007} (2007), no.~10, 029, 35.

\bibitem{Manin_motivic}
Ju.~I. Manin, {\em Correspondences, motifs and monoidal transformations}, Mat.
  Sb. (N.S.) {\bf 77 (119)} (1968), 475--507.

\bibitem{RefConifold}
Andrew Morrison, Sergey Mozgovoy, Kentaro Nagao, and Bal{{\'a}}zs
  Szendr{\H{o}}i, {\em Motivic {D}onaldson--{T}homas invariants of the conifold
  and the refined topological vertex}, Adv. Math. {\bf 230} (2012), no.~4-6,
  2065--2093.

\bibitem{MorrNagao}
Andrew Morrison and Kentaro Nagao, {\em Motivic {D}onaldson--{T}homas
  invariants of small crepant resolutions}, Algebra Number Theory {\bf 9}
  (2015), no.~4, 767--813.

\bibitem{Mozgovoy_Framed_WC}
Sergey Mozgovoy, {\em Wall-crossing formulas for framed objects}, Q. J. Math.
  {\bf 64} (2013), no.~2, 489--513.

\bibitem{Nagao_small_NCCR}
Kentaro Nagao, {\em Derived categories of small toric {C}alabi-{Y}au 3-folds
  and curve counting invariants}, Q. J. Math. {\bf 63} (2012), no.~4,
  965--1007.

\bibitem{NagNak}
Kentaro {Nagao} and Hiraku {Nakajima}, {\em {Counting invariant of perverse
  coherent sheaves and its wall-crossing}}, {Int. Math. Res. Not.} {\bf 2011}
  (2011), no.~17, 3885--3938.

\bibitem{Magnificent_colors}
Nikita Nekrasov and Nicol{\`o} Piazzalunga, {\em Magnificent four with colors},
  Comm. Math. Phys. {\bf 372} (2019), no.~2, 573--597.

\bibitem{Reineke_inventiones}
Markus Reineke, {\em The {H}arder-{N}arasimhan system in quantum groups and
  cohomology of quiver moduli}, Invent. Math. {\bf 152} (2003), no.~2,
  349--368.

\bibitem{ThesisR}
Andrea~T. Ricolfi, {\em {Local Donaldson--Thomas invariants and their
  refinements}}, Ph.D. thesis, University of Stavanger, 2017.

\bibitem{Ricolfi2018}
Andrea~T. Ricolfi, {\em The {DT}/{PT} correspondence for smooth curves}, Math.
  Z. {\bf 290} (2018), no.~1-2, 699--710.

\bibitem{LocalDT}
Andrea~T. Ricolfi, {\em {Local contributions to Donaldson--Thomas invariants}},
  Int. Math. Res. Not. IMRN {\bf 2018} (2018), no.~19, 5995--6025.

\bibitem{ricolfi2019motive}
Andrea~T. Ricolfi, {\em {On the motive of the Quot scheme of finite quotients
  of a locally free sheaf}}, J. Math. Pures Appl. {\bf 144} (2020), 50--68.

\bibitem{Quot19}
Andrea~T. Ricolfi, {\em {Virtual classes and virtual motives of Quot schemes on
  threefolds}}, Advances in Mathematics {\bf 369} (2020), 107182.

\bibitem{sheshmani}
Artan Sheshmani, {\em Higher rank stable pairs and virtual localization}, Comm.
  Anal. Geom. {\bf 24} (2016), no.~1, 139--193.

\end{thebibliography}

\bigskip

\medskip
\noindent
\emph{Andrea T.~Ricolfi}, \texttt{aricolfi@sissa.it} \\
\textsc{Scuola Internazionale Superiore di Studi Avanzati (SISSA), Via Bonomea 265, 34136 Trieste, Italy}

\medskip
\noindent
\emph{Alberto Cazzaniga}, \texttt{albe.cazzaniga@gmail.com} \\
\textsc{Area Science Park - Istituto Ricerca e Tecnologie, Padriciano 99, 34149 Trieste, Italy}

\end{document}